\documentclass[11pt, letterpaper,reqno]{amsart}
\usepackage[left=1in,right=1in,bottom=0.5in,top=0.8in]{geometry}
\usepackage{amsfonts}
\usepackage{amsmath, amssymb}
\usepackage{graphicx}
\usepackage[font=small,labelfont=bf]{caption}
\usepackage{epstopdf}
\usepackage{xcolor}
\usepackage{amsthm}
\usepackage{float}
\usepackage{pgfplots}
\usepackage{listings}
\usepackage{longtable}
\usepackage{mathrsfs}
\usepackage[pdfpagelabels,hyperindex]{hyperref}
\hypersetup{linkbordercolor=green}

\makeatletter
\newcommand*{\rom}[1]{\expandafter\@slowromancap\romannumeral #1@}
\makeatother

\newtheorem{thm}{Theorem}[section]

\newtheorem{lem}[thm]{Lemma}
\newtheorem{cond}[thm]{Condition}
\newtheorem{conj}[thm]{Conjecture}

\theoremstyle{definition}

\newtheorem{con}[thm]{Condition}

\theoremstyle{remark}

\colorlet{shadecolor}{gray!20}
\pgfplotsset{compat=1.9}
\def \p {\mathbb{P}}
\def \c {\mathbb{C}}
\def \p {\mathbb{P}}
\def \a {\alpha}
\def \z {\mathbb{Z}}
\def \pa {\partial}
\def \r {\mathbb{R}}
\def \o {\mathcal{O}}
\def \con {\text{con}}
\usepgflibrary{fpu}
\makeatletter
\let\c@equation\c@thm
\let\NAT@parse\undefined
\makeatother

\title{Gamma conjecture \rom{1} for blowing up $\p^n$ along $\p^r$}
\author{Zongrui Yang}
\address{Department of Mathematics\\
                Columbia University\\
New York, NY, 10027}
\email{zy2417@columbia.edu}
\begin{document}
\begin{abstract}
Consider the Fano manifold $X$ formed by blowing up $\p^n$ along its linear subspace $\p^r$, we check the conifold conditions \cite{GI, G} for its mirror Laurent polynomial $f$, which can imply that $X$ satisfies both Conjecture $\o$ and Gamma conjecture \rom{1} by Galkin-Golyshev-Iritani \cite{GGI}.
\end{abstract}
\maketitle
\section{Introduction and the Main result}\label{intro}
Let $X$ be a Fano manifold. Consider the unital commutative ring $(H(X),\star_0)$, where $H(X)=H^{even}(X)$ and $\star_0$ is the small quantum product. Galkin-Golyshev-Iritani \cite{GGI} conjectured that the eigenvalues of the linear operator $(c_1(X)\star_0): H(X)\rightarrow H(X)$ has some interesting properties:
\begin{conj}[Conjecture $\o$]
Let $\rho$ be the Fano index of $X$, and $$T=\sup\left\{|u|:u\in\text{Spec}(c_1(X)\star_0)\right\}$$ be the spectral radius of $(c_1(X)\star_0)$. Then
\begin{enumerate}
    \item [(1)] $T$ is an eigenvalue of $(c_1(X)\star_0)$ of multiplicity one
    \item [(2)] If $u$ is an eigenvalue of $(c_1(X)\star_0)$ with $|u|=T$ then $u=\zeta T$ for some $\zeta^\rho=1$
\end{enumerate}
\end{conj}
Gamma conjecture \rom{1} by Galkin-Golyshev-Iritani \cite{GGI} relates Gromov-Witten invariants to the topology of $X$. We briefly review one of its formulation as follows. Suppose $X$ satisfies conjecture $\o$, then the restriction $J_X(t)$ of Givental's big $J$-function to the anti-canonical line admits an asymptotic expansion at infinity, whose leading term $A_X\in H(X)$ is called the principal asymptotic class of $X$, which is conjectured to coincide (up to a scalar multiple) with the Gamma class 
$$\hat{\Gamma}_X=\prod_{i=1}^n\Gamma(1+x_i), \quad x_1,\dots,x_n \text{ are the Chern roots of }TX,\quad\Gamma(z)=\int_0^\infty e^tt^{z-1}dt$$
\begin{conj}[Gamma conjecture \rom{1}]
Suppose $X$ is a Fano manifold satisfying Conjecrue $\o$. Then there exists a constant $C\in\c$ such that $\hat{\Gamma}_X=CA_X$.
\end{conj}
Conjecture $\o$ and Gamma conjecture \rom{1} have been checked for several cases of Fano manifolds, including complex Grassmannians \cite{GGI}, del Pezzo surfaces \cite{HKLY}, projective complete intersections \cite{K} and Fano $3$-folds of Picard number one \cite{GZ}.

We review the results in Galkin-Iritani \cite[section 6]{GI}. Suppose $X$ is an $n$-dimensional Fano toric manifold. A mirror for $X$ is the Laurent polynomial \cite{Ghomo,Gmirror,HV}
$$
    f(x)=x^{b_1}+x^{b_2}+\dots+x^{b_m}
$$
where $x=(x_1,\dots,x_n)\in(\c^{*})^n$, $b_1,\dots,b_m\in\z^n$ are primitive generators of $1$-dimensional cones of the toric fan of $X$ and $x^{b_i}=x_1^{b_{i1}}\dots x_n^{b_{in}}$ for $b_i=(b_{i1},\dots,b_{in})$. By mirror symmetry \cite{Gmirror,I}, the small quantum cohomology ring $(H (X),\star_0)$ is isomorphic to the Jacobian ring of $f$,
 and the first Chern class $c_1(X)$ corresponds to the class of $f$ in Jacobian ring under this isomorphism. Therefore the set of eigenvalues of $(c_1(X)\star_0)$ coincides with the set of critical values of $f$, and their multiplicities coincide. Moreover, the restriction $f_{(\r_{>0})^n}$ is a real function on $(\r_{>0})^n$ that admits a global minimum at a unique point $x_{\con}\in(\r_{>0})^n$, called the conifold point \cite{GGI,GI,G}. 
 Hence the following conditions on the mirror Laurent polynomial $f$ implies conjecture $\o$ for $X$:
 \begin{cond}\label{cond}
Let $X$ be a Fano toric manifold with Fano index $\rho$ and $f$ be its mirror Laurent polynomial. Let $T_\con=f(x_\con)$ be the value of $f$ at the conifold point.
\begin{enumerate}
    \item [(1)] every critical value $u$ of $f$ satisfies $|u|\leq T_\con$
    \item [(2)] the conifold point is the unique critical point of $f$ contained in $f^{-1}(T_\con)$
    \item [(3)] for any critical value $u$ of $f$ with $|u|=T_\con$, we have $u=\zeta T_\con$, where $\zeta^\rho=1$
\end{enumerate}
\end{cond}
Galkin-Iritani \cite[section 6]{GI} proved that conditions $(1),(2)$ above imply Gamma conjecture \rom{1} for the toric Fano manifold $X$. 

From now on we suppose $X$ is the blow up of $\p^n$ along its linear subspace $\p^r$. From \cite{LCS}, $X$ is a toric Fano variety whose toric fan has $1$-dimensional cones generated by primitive generators
$$\left\{e_1,\dots,e_n,-(e_1+\dots+e_n),-(e_1+\dots+e_{r+1})\right\}\subset\z^n$$
where $e_i=(0,\dots,\overset{i\text{th}}{1},\dots,0)$, $i=1,\dots,n$.
Hence the mirror Laurent polynomial for $X$ is 
$$
    f(x_1,\dots,x_n)=x_1+\dots+x_n+\frac{1}{x_1\dots x_n}+\frac{1}{x_1\dots x_{r+1}}
$$
We denote $k=r+1$, $m=n-r-1$, then $m,k\geq 1$. We need to analyze the critical values of $f$. Solve the critical point equations $\frac{\pa f}{\pa x_i}=0$ for $i=1,\dots,m+k$, we get 
$$
    x_{k+1}=\dots=x_{m+k}\overset{\text{def}}{=}x
$$
$$
    x_1=\dots=x_k=x^{m+1}+x
$$
and $x$ has to satisfy
$$
    u(x)=x^{m+1}(x^{m+1}+x)^k-1=0
$$
The critical value is
$$
    f(x_1,\dots,x_n)=(k+1)x^{m+1}+(k+m+1)x\overset{\text{def}}{=}g(x)
$$
Denote the only positive real root of $u(x)$ by $r_+$. Then the conifold point is 
$$
    x_{\con}=(\underbrace{r_+^{m+1}+r_+,\dots,r_+^{m+1}+r_+}_{k},\underbrace{r_+,\dots,r_+}_{m})
$$
and the corresponding critical value is
$$
    T_{\con}=g(r_+)=(k+1)r_+^{m+1}+(k+m+1)r_+
$$
Those $g(\a)$ for complex roots $\a$ of $u(x)$ are precisely the critical values of $f$.
\begin{thm}\label{thm}
For any complex root $\a$ of $u(x)$ we have $|g(\a)|\leq T_{\con}$. The equality holds only if $\a=\zeta_m^dr_+$, where $\zeta_m=e^{\frac{2\pi i}{m}}$ and $d\in\z$, in which case $m|(k+1)d$ and $g(\a)=\zeta_m^dT_{\con}$.
\end{thm}
Remark that $K_X=-(m+k+1)i^*(H)+mE$, where $H$ is the hyperplane class on $\p^n$ and $E$ is the exceptional class. If the equality in theorem \ref{thm} holds we have $m|(k+1)d$, then it is easy to see $m|\rho d$, hence $(\zeta)^\rho=(\zeta_m^d)^\rho=1$, where $\rho$ is the Fano index of $X$. Hence condition \ref{cond} holds for the mirror Laurent polynomial $f$ of $X$. We conclude the following:
\begin{thm}
Conjecture $\o$ and Gamma conjecture \rom{1} hold for the blow up of $\p^n$ along $\p^r$.
\end{thm}
\section{The Proof of Theorem \ref{thm}}
In the rest of the paper we denote any complex root of $u(x)$ by $\a$.
Observe that since $u(x)$ is strictly increasing on $(0,\infty)$ and 
$$
    u(|\a|)+1=|\a|^{m+1}(|\a|^{m+1}+|\a|)^k\geq|\a^{m+1}(\a^{m+1}+\a)^k|=|u(\a)+1|=1=u(r_+)+1
$$
we have $|\a|\geq r_+$.

Define 
$$
    h(x)=(k+1)x^{-\frac{m+1}{k}}+mx, \quad x\in(0,\infty)
$$

then $h(x)$ is strictly decreasing on $(0,r_0)$ and strictly incrasing on $(r_0,\infty)$, where 
$$
    r_0=(\frac{(m+1)(k+1)}{mk})^{\frac{k}{m+k+1}}
$$

 It is clear that $r_0>1>r_+$. If the lemma below holds, then for any complex root $\a$ of $u(x)$, we have $r_+\leq|\a|\leq r_0$, so 
$$
    |g(\a)|\leq (k+1)|\a^{m+1}+\a|+m|\a|=
    (k+1)|\a|^{-\frac{m+1}{k}}+m|\a|=h(|\a|)\leq h(r_+)=g(r_+)=T_{\con}
$$
and we conclude the inequality part of the theorem.
\begin{lem}
Any complex root $\a$ of $u(x)$ satisfies $|\a|< r_0$.
\end{lem}
\begin{proof} 
We suppose on the contradiction that $|\a|\geq r_0>1$. Then
\begin{equation*}
\begin{split}
&1=|\a|^{m+1}|\a^{m+1}+\a|^k\\
&\geq|\a|^{m+k+1}(|\a|^m-1)^k\\
&\geq r_0^{m+k+1}(r_0^m-1)^k\\
&=(\frac{(m+1)(k+1)}{mk})^{k}((\frac{(m+1)(k+1)}{mk})^{\frac{km}{m+k+1}}-1)^k
\end{split}
\end{equation*}
Take exponential $\frac{1}{k}$ on both side, we get
$$
1\geq\frac{(m+1)(k+1)}{mk}((\frac{(m+1)(k+1)}{mk})^{\frac{km}{m+k+1}}-1)
$$
We divide into four cases
\begin{enumerate}
    \item [(\rom{1})] Case $(m-1)(k-1)\geq 2$. 
We have 
\begin{equation*}
    \begin{split}
        &1\geq\frac{(m+1)(k+1)}{mk}((\frac{(m+1)(k+1)}{mk})^{\frac{km}{m+k+1}}-1)\\
        &\geq \frac{(m+1)(k+1)}{mk}\frac{km}{m+k+1}\frac{m+k+1}{mk}>1
    \end{split}
\end{equation*}
contradiction, where we used $(1+x)^\gamma\geq 1+x\gamma$ for $\gamma\geq 1$.
    \item[(\rom{2})] Case $(m,k)=(2,2)$. We have
    $$1\geq(\frac{9}{4})((\frac{9}{4})^{\frac{4}{5}}-1)>1$$
    contradiction.
    \item [(\rom{3})]Case $m=1$, $k\geq 2$. We have $$1\geq\frac{2(k+1)}{k}(
    (\frac{2(k+1)}{k})^\frac{k}{k+2}-1)>1$$
    contradiction, where we directly check the inequality for $1\leq k\leq 5$, and for $k\geq 6$ we use
    $$\frac{2(k+1)}{k}(
    (\frac{2(k+1)}{k})^\frac{k}{k+2}-1)>2(2^{\frac{k}{k+2}}-1)>1$$
    \item [(\rom{4})] Case $k=1$. We have $u(x)=x^{2m+2}+x^{m+2}-1$ and $r_0=(\frac{2(m+1)}{m})^{\frac{1}{m+2}}$. Define $$v(x)=x^{2m+2}-x^{m+2}-1$$ then $v(x)$ is strictly increasing on $(1,\infty)$ and has a unique root $r_-$ in $(1,\infty)$. We again suppose on the contradiction that $|\a|\geq r_0>1$. Then $$1=|\a|^{m+2}|\a^m+1|\geq |\a|^{m+2}(|\a|^m-1)=v(|\a|)+1$$
    hence $|\a|\leq r_-$, in particular $r_0\leq r_-$. But on the other hand we have 
    $$v(r_0)=(\frac{2(m+1)}{m})^{\frac{2m+2}{m+2}}-\frac{2(m+1)}{m}-1>0=v(r_-)$$
    where we directly check the inequality for $m=1,2$ and for $m\geq 3$ we use $$v(r_0)>v(2^{\frac{1}{m+2}})=2^{\frac{2m+2}{m+2}}-3>0$$ 
    Hence $r_-<r_0$ by $v(x)$ being strictly increasing on $(1,\infty)$, contradiction.
\end{enumerate}
\end{proof}
Next we analyze the equality conditions. This happens when $|\a|=r_+$. Plug it in $\a^{m+1}(\a^{m+1}+\a)^k=1=r_+^{m+1}(r_+^{m+1}+r_+)^k$ we know that $|\a^{m+1}+\a|=r_+^{m+1}+r_+$, hence $\a^{m+1}$ differs by a positive real multiple with $\a$, therefore $\a=\zeta_m^dr_+$ for some $d\in\z$, where $\zeta_m=e^{\frac{2\pi i}{m}}$. We again plug it in $\a^{m+1}(\a^{m+1}+\a)^k=1$ and get $m|(k+1)d$. In this case $g(\a)=(k+1)\a^{m+1}+(k+m+1)\a=\zeta_m^dg(r_+)=\zeta_m^dT_{\con}$. The theorem is proved.
\section{Acknowledgement}
We acknowledge Professor Gang Tian for introducing me to Gamma conjecture and Yihan Wang for some discussions in the proof.

\end{document}